\documentclass{article}%
\usepackage{amsfonts}
\usepackage{amssymb}
\usepackage{amsmath}
\usepackage{amsthm}
\usepackage{fullpage}
\usepackage{xcolor}
\usepackage{graphicx}
\usepackage{graphics}
\usepackage{pstricks}
\usepackage{pst-node}
\usepackage{pst-plot}
\usepackage{slashbox}
\usepackage{float}%
\providecommand{\U}[1]{\protect\rule{.1in}{.1in}}
\allowdisplaybreaks
\newtheorem{theorem}{Theorem}
\newtheorem{lemma}{Lemma}

\newtheorem{example}{Example}
\newtheorem{remark}{Remark}
\newtheorem{notation}{Notation}

\begin{document}

\title{
Block Structure of Cyclotomic Polynomials}
\author{Ala'a Al-Kateeb, Hoon Hong, Eunjeong Lee}
\date{\today}
\maketitle

\begin{abstract}
In this paper, we list  several interesting  structures of  cyclotomic polynomials:   specifically relations among blocks obtained by suitable  partition of cyclotomic polynomials. 
 We present explicit and self-contained proof for all
of them, using a uniform terminology and technique.
\end{abstract}

\section{Introduction}

The cyclotomic polynomial $\Phi_{n}(x)$ is the monic polynomial in
$\mathbb{Z}[x]$ whose zeros are the primitive $n$-th roots of unity. It has
numerous application in many areas of mathematics, science and engineering.
Thus it is important to understand its structure.

In this paper, we list  several interesting  structures of  cyclotomic polynomials:   specifically relations among blocks obtained by suitable  partition of cyclotomic polynomials. 
  Some of the structures are already known or at
least implicit in the previous literature~\cite{KA1,KA2,AM2}. Others are new,
as far as we are aware.  We present explicit and self-contained proof for all
of them, using a uniform terminology.

Let $m$ be a fixed odd square-free positive integer and $p$ be a prime number
relatively prime to $m$. Let $f_{m,p,i}$ be the $i$-th``block" of $\Phi_{mp}$
in the radix $x^{p}$. Let $f_{m,p,i,j}$ be the $j$-th ``block'' of $f_{m,p,i}$
in the radix $x^{m}$. We provide the following types of structural results.

\begin{itemize}
\item Theorem \ref{thm:explicit} provides an explicit formula for $f_{m,p,i,j}$.

\item Theorem \ref{thm:intra-str} lists several ``intra'' structures of
$f_{m,p,i,j}$, that is the relations among them with same~$p$ but different
$i$ and $j$.

\item Theorem \ref{thm:inter-str} lists several `inter" properties of
$f_{m,p,i,j}$, that is the relation among them with different~$p$, but with
same $i$ and $j$.
\end{itemize}

\section{Block Structures}

\begin{notation}
[Partition]\label{notation:partition}Let%
\begin{align*}
\Phi_{mp}\left(  x\right)   &  =\sum_{i\geq0}f_{m,p,i}\left(  x\right)
\ x^{ip} &  &  \text{where }\deg f_{m,p,i}\left(  x\right)  <p\\
f_{m,p,i}\left(  x\right)   &  =\sum_{j\geq0}f_{m,p,i,j}\left(  x\right)
x^{jm} &  &  \text{where }\deg f_{m,p,i,j}\left(  x\right)  <m
\end{align*}

\end{notation}

\begin{example}
[Partition]We will visualize a polynomial by a graph where the horizonal axis
stands for the exponents and the vertical axis stands for the corresponding
coefficients.Let $m=15$ and $p=53$. Then $\phi(m)-1=7$, $q = 3$ and $r=8$.The partition of the list of the coefficients of $\Phi_{mp}$ into $f_{m,p,i,j}$'s is illustrated by the following diagram.{\tiny\psset{unit=0.15}
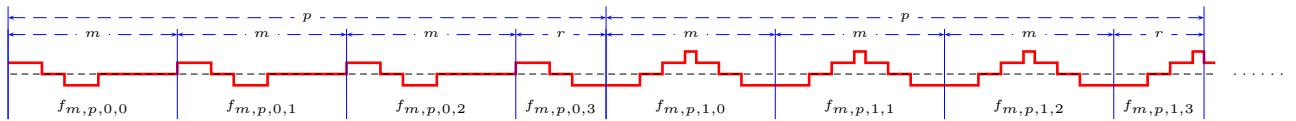
\[\pspicture[shift=*](-0.5,-4)(105.5,6)\psset{linewidth=1pt,linecolor=red}\psline{-}(-0.5,1.0)(0.5,1.0)(0.5,1.0)(1.5,1.0)(1.5,1.0)(2.5,1.0)(2.5,0.0)(3.5,0.0)(3.5,0.0)(4.5,0.0)(4.5,-1.0)(5.5,-1.0)(5.5,-1.0)(6.5,-1.0)(6.5,-1.0)(7.5,-1.0)(7.5,0.0)(8.5,0.0)(8.5,0.0)(9.5,0.0)(9.5,0.0)(10.5,0.0)(10.5,0.0)(11.5,0.0)(11.5,0.0)(12.5,0.0)(12.5,0.0)(13.5,0.0)(13.5,0.0)(14.5,0.0)(14.5,1.0)(15.5,1.0)(15.5,1.0)(16.5,1.0)(16.5,1.0)(17.5,1.0)(17.5,0.0)(18.5,0.0)(18.5,0.0)(19.5,0.0)(19.5,-1.0)(20.5,-1.0)(20.5,-1.0)(21.5,-1.0)(21.5,-1.0)(22.5,-1.0)(22.5,0.0)(23.5,0.0)(23.5,0.0)(24.5,0.0)(24.5,0.0)(25.5,0.0)(25.5,0.0)(26.5,0.0)(26.5,0.0)(27.5,0.0)(27.5,0.0)(28.5,0.0)(28.5,0.0)(29.5,0.0)(29.5,1.0)(30.5,1.0)(30.5,1.0)(31.5,1.0)(31.5,1.0)(32.5,1.0)(32.5,0.0)(33.5,0.0)(33.5,0.0)(34.5,0.0)(34.5,-1.0)(35.5,-1.0)(35.5,-1.0)(36.5,-1.0)(36.5,-1.0)(37.5,-1.0)(37.5,0.0)(38.5,0.0)(38.5,0.0)(39.5,0.0)(39.5,0.0)(40.5,0.0)(40.5,0.0)(41.5,0.0)(41.5,0.0)(42.5,0.0)(42.5,0.0)(43.5,0.0)(43.5,0.0)(44.5,0.0)(44.5,1.0)(45.5,1.0)(45.5,1.0)(46.5,1.0)(46.5,1.0)(47.5,1.0)(47.5,0.0)(48.5,0.0)(48.5,0.0)(49.5,0.0)(49.5,-1.0)(50.5,-1.0)(50.5,-1.0)(51.5,-1.0)(51.5,-1.0)(52.5,-1.0)(52.5,-1.0)(53.5,-1.0)(53.5,-1.0)(54.5,-1.0)(54.5,-1.0)(55.5,-1.0)(55.5,0.0)(56.5,0.0)(56.5,0.0)(57.5,0.0)(57.5,1.0)(58.5,1.0)(58.5,1.0)(59.5,1.0)(59.5,2.0)(60.5,2.0)(60.5,1.0)(61.5,1.0)(61.5,1.0)(62.5,1.0)(62.5,0.0)(63.5,0.0)(63.5,0.0)(64.5,0.0)(64.5,-1.0)(65.5,-1.0)(65.5,-1.0)(66.5,-1.0)(66.5,-1.0)(67.5,-1.0)(67.5,-1.0)(68.5,-1.0)(68.5,-1.0)(69.5,-1.0)(69.5,-1.0)(70.5,-1.0)(70.5,0.0)(71.5,0.0)(71.5,0.0)(72.5,0.0)(72.5,1.0)(73.5,1.0)(73.5,1.0)(74.5,1.0)(74.5,2.0)(75.5,2.0)(75.5,1.0)(76.5,1.0)(76.5,1.0)(77.5,1.0)(77.5,0.0)(78.5,0.0)(78.5,0.0)(79.5,0.0)(79.5,-1.0)(80.5,-1.0)(80.5,-1.0)(81.5,-1.0)(81.5,-1.0)(82.5,-1.0)(82.5,-1.0)(83.5,-1.0)(83.5,-1.0)(84.5,-1.0)(84.5,-1.0)(85.5,-1.0)(85.5,0.0)(86.5,0.0)(86.5,0.0)(87.5,0.0)(87.5,1.0)(88.5,1.0)(88.5,1.0)(89.5,1.0)(89.5,2.0)(90.5,2.0)(90.5,1.0)(91.5,1.0)(91.5,1.0)(92.5,1.0)(92.5,0.0)(93.5,0.0)(93.5,0.0)(94.5,0.0)(94.5,-1.0)(95.5,-1.0)(95.5,-1.0)(96.5,-1.0)(96.5,-1.0)(97.5,-1.0)(97.5,-1.0)(98.5,-1.0)(98.5,-1.0)(99.5,-1.0)(99.5,-1.0)(100.5,-1.0)(100.5,0.0)(101.5,0.0)(101.5,0.0)(102.5,0.0)(102.5,1.0)(103.5,1.0)(103.5,1.0)(104.5,1.0)(104.5,2.0)(105.5,2.0)(105.5,1.0)(106.5,1.0)
\psset{linewidth=0.1pt,linecolor=black}\psline{-}(-0.3,0.0)(0.3,0.0)\psline{-}(0.7,0.0)(1.3,0.0)\psline{-}(1.7,0.0)(2.3,0.0)\psline{-}(2.7,0.0)(3.3,0.0)\psline{-}(3.7,0.0)(4.3,0.0)\psline{-}(4.7,0.0)(5.3,0.0)\psline{-}(5.7,0.0)(6.3,0.0)\psline{-}(6.7,0.0)(7.3,0.0)\psline{-}(7.7,0.0)(8.3,0.0)\psline{-}(8.7,0.0)(9.3,0.0)\psline{-}(9.7,0.0)(10.3,0.0)\psline{-}(10.7,0.0)(11.3,0.0)\psline{-}(11.7,0.0)(12.3,0.0)\psline{-}(12.7,0.0)(13.3,0.0)\psline{-}(13.7,0.0)(14.3,0.0)\psline{-}(14.7,0.0)(15.3,0.0)\psline{-}(15.7,0.0)(16.3,0.0)\psline{-}(16.7,0.0)(17.3,0.0)\psline{-}(17.7,0.0)(18.3,0.0)\psline{-}(18.7,0.0)(19.3,0.0)\psline{-}(19.7,0.0)(20.3,0.0)\psline{-}(20.7,0.0)(21.3,0.0)\psline{-}(21.7,0.0)(22.3,0.0)\psline{-}(22.7,0.0)(23.3,0.0)\psline{-}(23.7,0.0)(24.3,0.0)\psline{-}(24.7,0.0)(25.3,0.0)\psline{-}(25.7,0.0)(26.3,0.0)\psline{-}(26.7,0.0)(27.3,0.0)\psline{-}(27.7,0.0)(28.3,0.0)\psline{-}(28.7,0.0)(29.3,0.0)\psline{-}(29.7,0.0)(30.3,0.0)\psline{-}(30.7,0.0)(31.3,0.0)\psline{-}(31.7,0.0)(32.3,0.0)\psline{-}(32.7,0.0)(33.3,0.0)\psline{-}(33.7,0.0)(34.3,0.0)\psline{-}(34.7,0.0)(35.3,0.0)\psline{-}(35.7,0.0)(36.3,0.0)\psline{-}(36.7,0.0)(37.3,0.0)\psline{-}(37.7,0.0)(38.3,0.0)\psline{-}(38.7,0.0)(39.3,0.0)\psline{-}(39.7,0.0)(40.3,0.0)\psline{-}(40.7,0.0)(41.3,0.0)\psline{-}(41.7,0.0)(42.3,0.0)\psline{-}(42.7,0.0)(43.3,0.0)\psline{-}(43.7,0.0)(44.3,0.0)\psline{-}(44.7,0.0)(45.3,0.0)\psline{-}(45.7,0.0)(46.3,0.0)\psline{-}(46.7,0.0)(47.3,0.0)\psline{-}(47.7,0.0)(48.3,0.0)\psline{-}(48.7,0.0)(49.3,0.0)\psline{-}(49.7,0.0)(50.3,0.0)\psline{-}(50.7,0.0)(51.3,0.0)\psline{-}(51.7,0.0)(52.3,0.0)\psline{-}(52.7,0.0)(53.3,0.0)\psline{-}(53.7,0.0)(54.3,0.0)\psline{-}(54.7,0.0)(55.3,0.0)\psline{-}(55.7,0.0)(56.3,0.0)\psline{-}(56.7,0.0)(57.3,0.0)\psline{-}(57.7,0.0)(58.3,0.0)\psline{-}(58.7,0.0)(59.3,0.0)\psline{-}(59.7,0.0)(60.3,0.0)\psline{-}(60.7,0.0)(61.3,0.0)\psline{-}(61.7,0.0)(62.3,0.0)\psline{-}(62.7,0.0)(63.3,0.0)\psline{-}(63.7,0.0)(64.3,0.0)\psline{-}(64.7,0.0)(65.3,0.0)\psline{-}(65.7,0.0)(66.3,0.0)\psline{-}(66.7,0.0)(67.3,0.0)\psline{-}(67.7,0.0)(68.3,0.0)\psline{-}(68.7,0.0)(69.3,0.0)\psline{-}(69.7,0.0)(70.3,0.0)\psline{-}(70.7,0.0)(71.3,0.0)\psline{-}(71.7,0.0)(72.3,0.0)\psline{-}(72.7,0.0)(73.3,0.0)\psline{-}(73.7,0.0)(74.3,0.0)\psline{-}(74.7,0.0)(75.3,0.0)\psline{-}(75.7,0.0)(76.3,0.0)\psline{-}(76.7,0.0)(77.3,0.0)\psline{-}(77.7,0.0)(78.3,0.0)\psline{-}(78.7,0.0)(79.3,0.0)\psline{-}(79.7,0.0)(80.3,0.0)\psline{-}(80.7,0.0)(81.3,0.0)\psline{-}(81.7,0.0)(82.3,0.0)\psline{-}(82.7,0.0)(83.3,0.0)\psline{-}(83.7,0.0)(84.3,0.0)\psline{-}(84.7,0.0)(85.3,0.0)\psline{-}(85.7,0.0)(86.3,0.0)\psline{-}(86.7,0.0)(87.3,0.0)\psline{-}(87.7,0.0)(88.3,0.0)\psline{-}(88.7,0.0)(89.3,0.0)\psline{-}(89.7,0.0)(90.3,0.0)\psline{-}(90.7,0.0)(91.3,0.0)\psline{-}(91.7,0.0)(92.3,0.0)\psline{-}(92.7,0.0)(93.3,0.0)\psline{-}(93.7,0.0)(94.3,0.0)\psline{-}(94.7,0.0)(95.3,0.0)\psline{-}(95.7,0.0)(96.3,0.0)\psline{-}(96.7,0.0)(97.3,0.0)\psline{-}(97.7,0.0)(98.3,0.0)\psline{-}(98.7,0.0)(99.3,0.0)\psline{-}(99.7,0.0)(100.3,0.0)\psline{-}(100.7,0.0)(101.3,0.0)\psline{-}(101.7,0.0)(102.3,0.0)\psline{-}(102.7,0.0)(103.3,0.0)\psline{-}(103.7,0.0)(104.3,0.0)\psline{-}(104.7,0.0)(105.3,0.0)\psline{-}(105.7,0.0)(106.3,0.0)\rput(7.0,-3.0){$f_{m,p,0,0}$}\psset{linewidth=0.1pt,linecolor=blue}\psline{-}(-0.5,-4.0)(-0.5,4.0)\pcline[linestyle=dashed]{<->}(-0.5,3.5)(14.5,3.5)\ncput*{$m$}\rput(22.0,-3.0){$f_{m,p,0,1}$}\psset{linewidth=0.1pt,linecolor=blue}\psline{-}(14.5,-4.0)(14.5,4.0)\pcline[linestyle=dashed]{<->}(14.5,3.5)(29.5,3.5)\ncput*{$m$}\rput(37.0,-3.0){$f_{m,p,0,2}$}\psset{linewidth=0.1pt,linecolor=blue}\psline{-}(29.5,-4.0)(29.5,4.0)\pcline[linestyle=dashed]{<->}(29.5,3.5)(44.5,3.5)\ncput*{$m$}\rput(48.5,-3.0){$f_{m,p,0,3}$}\psset{linewidth=0.1pt,linecolor=blue}\psline{-}(44.5,-4.0)(44.5,4.0)\pcline[linestyle=dashed]{<->}(44.5,3.5)(52.5,3.5)\ncput*{$r$}\psset{linewidth=0.1pt,linecolor=blue}\psline{-}(-0.5,-4.0)(-0.5,6.0)\pcline[linestyle=dashed]{<->}(-0.5,5.0)(52.5,5.0)\ncput*{$p$}\rput(60.0,-3.0){$f_{m,p,1,0}$}\psset{linewidth=0.1pt,linecolor=blue}\psline{-}(52.5,-4.0)(52.5,4.0)\pcline[linestyle=dashed]{<->}(52.5,3.5)(67.5,3.5)\ncput*{$m$}\rput(75.0,-3.0){$f_{m,p,1,1}$}\psset{linewidth=0.1pt,linecolor=blue}\psline{-}(67.5,-4.0)(67.5,4.0)\pcline[linestyle=dashed]{<->}(67.5,3.5)(82.5,3.5)\ncput*{$m$}\rput(90.0,-3.0){$f_{m,p,1,2}$}\psset{linewidth=0.1pt,linecolor=blue}\psline{-}(82.5,-4.0)(82.5,4.0)\pcline[linestyle=dashed]{<->}(82.5,3.5)(97.5,3.5)\ncput*{$m$}\rput(101.5,-3.0){$f_{m,p,1,3}$}\psset{linewidth=0.1pt,linecolor=blue}\psline{-}(97.5,-4.0)(97.5,4.0)\pcline[linestyle=dashed]{<->}(97.5,3.5)(105.5,3.5)\ncput*{$r$}\psset{linewidth=0.1pt,linecolor=blue}\psline{-}(52.5,-4.0)(52.5,6.0)\pcline[linestyle=dashed]{<->}(52.5,5.0)(105.5,5.0)\ncput*{$p$}\psset{linewidth=0.1pt,linecolor=blue}\psline{-}(105.5,-4.0)(105.5,6.0)\rput[l](108.0,0){$\ldots\ldots$}\endpspicture \]} 
\end{example}

\noindent\noindent Let ${\mathrm{rem}}$ stand for polynomial remainder or
integer remainder. For integer remainder, we will require that the remainder
is non-negative. We need to define some operations on $f_{m,p,i,j}$'s.

\begin{notation}
[Operation]\label{not:operation} For a polynomial $f$ of degree less than $m$, let

\begin{enumerate}
\item $\mathcal{T}_{s}f={\mathrm{rem}}(f,x^{s})$ \hspace{9em}
\textquotedblleft Truncate\textquotedblright

\item $\mathcal{F}f=x^{m-1}f\left(  x\newline^{-1}\right)  $ \hspace{8em}
\textquotedblleft Flip\textquotedblright

\item $\mathcal{R}_{s}f={\mathrm{rem}}(x^{m-\operatorname*{rem}\left(
s,m\right)  }f,x^{m}-1)$ \hspace{1.6em} \textquotedblleft
Rotate\textquotedblright

\item $\mathcal{E}_{s}f=f(x^{\operatorname*{rem}\left(  s,m\right)  })$
\hspace{8.3em} \textquotedblleft Expand\textquotedblright
\end{enumerate}
\end{notation}

\begin{example}
[Operation]Let $f=1+2x+3x^{2}+4x^{3}+5x^{4}$ and $m=5$. Note \psset{unit=0.3}\[
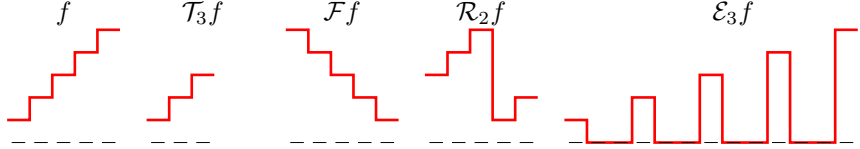
\begin{array}{cccccc}
f & \mathcal{T}_{3} f & \mathcal{F} f & \mathcal{R}_{2} f &\mathcal{E}_{3} f \\ 
\pspicture[shift=*](-0.5,-5)(4.5,5)\psset{linewidth=1pt,linecolor=red}\psline{-}(-0.5,1.0)(0.5,1.0)(0.5,2.0)(1.5,2.0)(1.5,3.0)(2.5,3.0)(2.5,4.0)(3.5,4.0)(3.5,5.0)(4.5,5.0)\psset{linewidth=0.1pt,linecolor=black}\psline{-}(-0.3,0.0)(0.3,0.0)\psline{-}(0.7,0.0)(1.3,0.0)\psline{-}(1.7,0.0)(2.3,0.0)\psline{-}(2.7,0.0)(3.3,0.0)\psline{-}(3.7,0.0)(4.3,0.0)\endpspicture  & 
\pspicture[shift=*](-0.5,-5)(4.5,5)\psset{linewidth=1pt,linecolor=red}\psline{-}(-0.5,1.0)(0.5,1.0)(0.5,2.0)(1.5,2.0)(1.5,3.0)(2.5,3.0)\psset{linewidth=0.1pt,linecolor=black}\psline{-}(-0.3,0.0)(0.3,0.0)\psline{-}(0.7,0.0)(1.3,0.0)\psline{-}(1.7,0.0)(2.3,0.0)\endpspicture  & 
\pspicture[shift=*](-0.5,-5)(4.5,5)\psset{linewidth=1pt,linecolor=red}\psline{-}(-0.5,5.0)(0.5,5.0)(0.5,4.0)(1.5,4.0)(1.5,3.0)(2.5,3.0)(2.5,2.0)(3.5,2.0)(3.5,1.0)(4.5,1.0)\psset{linewidth=0.1pt,linecolor=black}\psline{-}(-0.3,0.0)(0.3,0.0)\psline{-}(0.7,0.0)(1.3,0.0)\psline{-}(1.7,0.0)(2.3,0.0)\psline{-}(2.7,0.0)(3.3,0.0)\psline{-}(3.7,0.0)(4.3,0.0)\endpspicture  & 
\pspicture[shift=*](-0.5,-5)(4.5,5)\psset{linewidth=1pt,linecolor=red}\psline{-}(-0.5,3.0)(0.5,3.0)(0.5,4.0)(1.5,4.0)(1.5,5.0)(2.5,5.0)(2.5,1.0)(3.5,1.0)(3.5,2.0)(4.5,2.0)\psset{linewidth=0.1pt,linecolor=black}\psline{-}(-0.3,0.0)(0.3,0.0)\psline{-}(0.7,0.0)(1.3,0.0)\psline{-}(1.7,0.0)(2.3,0.0)\psline{-}(2.7,0.0)(3.3,0.0)\psline{-}(3.7,0.0)(4.3,0.0)\endpspicture  & 
\pspicture[shift=*](-0.5,-5)(14.5,5)\psset{linewidth=1pt,linecolor=red}\psline{-}(-0.5,1.0)(0.5,1.0)(0.5,0.0)(1.5,0.0)(1.5,0.0)(2.5,0.0)(2.5,2.0)(3.5,2.0)(3.5,0.0)(4.5,0.0)(4.5,0.0)(5.5,0.0)(5.5,3.0)(6.5,3.0)(6.5,0.0)(7.5,0.0)(7.5,0.0)(8.5,0.0)(8.5,4.0)(9.5,4.0)(9.5,0.0)(10.5,0.0)(10.5,0.0)(11.5,0.0)(11.5,5.0)(12.5,5.0)\psset{linewidth=0.1pt,linecolor=black}\psline{-}(-0.3,0.0)(0.3,0.0)\psline{-}(0.7,0.0)(1.3,0.0)\psline{-}(1.7,0.0)(2.3,0.0)\psline{-}(2.7,0.0)(3.3,0.0)\psline{-}(3.7,0.0)(4.3,0.0)\psline{-}(4.7,0.0)(5.3,0.0)\psline{-}(5.7,0.0)(6.3,0.0)\psline{-}(6.7,0.0)(7.3,0.0)\psline{-}(7.7,0.0)(8.3,0.0)\psline{-}(8.7,0.0)(9.3,0.0)\psline{-}(9.7,0.0)(10.3,0.0)\psline{-}(10.7,0.0)(11.3,0.0)\psline{-}(11.7,0.0)(12.3,0.0)\endpspicture \end{array}
\]
\end{example}

\noindent Throughout this paper, for an integer $m$ and a prime $p$, we denote%
\[
r:=\mathrm{rem}\left(  p,m\right)  ,\;\;q:=\mathrm{quo}(p,m)
\]
where $\mathrm{quo,}$ of course, stand for quotient. Now we list several
structures: some known and some new.

\begin{theorem}
[Explicit]\label{thm:explicit} For $0\leq i\leq\varphi(m)-1$ and $0\leq j\leq q$,
\[
f_{m,p,i,j}=%
\begin{cases}
-\mathcal{R}_{ir}(\Psi_{m}\cdot\mathcal{E}_{r}\mathcal{T}_{i+1}\Phi_{m}) &
0\leq j\leq q-1\\
~~\mathcal{T}_{r}f_{m,p,i,0} & j=q
\end{cases}
\]
where $\Psi_{m}\left(  x\right)  =\frac{x^{m}-1}{\Phi_{m}\left(  x\right)  }$,
the $m$-th inverse cyclotomic polynomial.
\end{theorem}


\begin{theorem}
[Intra-Structure]\ \label{thm:intra-str} Within a cyclotomic polynomial, we have

\begin{enumerate}
\item (Repetition)\ \ $f_{m,p,i,0}=\cdots=f_{m,p,i,q-1}$

\item (Truncation) $f_{m,p,i,q}=\mathcal{T}_{r}f_{m,p,i,0}$

\item (Symmetry)\ $f_{m,p,i^{\prime},0}=\mathcal{R}_{\varphi(m)-1-r}%
\mathcal{F}f_{m,p,i,0}$ \hspace{2.4em} if $i^{\prime}+i=\varphi(m)-1$
\end{enumerate}
\end{theorem}

\begin{remark}
The \textquotedblleft repetition\textquotedblright\ structure was observed by
Arnold and Monaga \cite{AM2} in the context of computing cyclotomic
polynomials. In particular, Algorithm 7 in their paper exploited the
repetition structure to improve the time and space complexity when $p\gg m$.
\end{remark}

\begin{theorem}
[Inter-Structure]\ \label{thm:inter-str} Among cyclotomic polynomials, we have

\begin{enumerate}
\item (Invariance)\ \ \ \ \ \ \ \ \ $f_{m,\tilde{p},i,0}=~~f_{m,p,i,0}$
\hspace{4.8em} if $\tilde{p}-p\equiv_{m}0$

\item (Semi-Invariance)\ \ $f_{m,\tilde{p},i,0}=-\mathcal{R}_{\varphi
(m)-1}\mathcal{F}f_{m,p,i,0}\;\;$ if $\tilde{p}+p\equiv_{m}0$
\end{enumerate}
\end{theorem}

\begin{remark}
The \textquotedblleft invariance\textquotedblright\ structure was observed by
Kaplan \cite{KA1,KA2} in the context of studying flat cyclotomic polynomials.
In particular, The proof of  \cite[Theorem 4]{KA2} used invariance structure.
Furthermore, the proof of  \cite[Theorem 3]{KA1} used the semi-invariance
structure for $\Phi_{p_{1}p_{2}p_{3}}$.
\end{remark}

\noindent Now we present a set of examples to illustrate the structural theorems.

\begin{example}
[Repetition]\input{example_str_repetition}
\end{example}

\begin{example}
[Truncation]\input{example_str_truncation}
\end{example}

\begin{example}
[Symmetry]\input{example_str_symmetry}
\end{example}

\begin{example}
[Invariance]\input{example_str_invariance}
\end{example}

\begin{example}
[Semi-Invariance]\input{example_str_semi_invariance}
\end{example}

\section{Proof of Block Structures}

In this section, we prove all the theorems given in
the previous section.

\begin{lemma}
\label{lem:operation} Let $u,v$ be integers. For polynomials $f$ and
$g=\sum_{t\geq0}c_{t}x^t, $ with degree less than $m$, we have

\begin{enumerate}
\item $\mathcal{R}_{u-v}f=\mathcal{R}_{u}\left(  x^{\mathrm{rem}%
(v,m)}f\right)  $

\item $\sum_{t}c_{t}\mathcal{R}_{u-tv}f=\mathcal{R}_{u}\left(  f\cdot
\mathcal{E}_{v}\sum_{t} c_{t}x^{t}\right)  $

\end{enumerate}
\end{lemma}

\begin{proof} We prove each one by one.
\begin{enumerate}
\item Immediate from Notation~\ref{not:operation}.

\item Note
\begin{align*}
\sum_{t}c_{t}\mathcal{R}_{u-tv}f  &  = \mathrm{rem}\left(  \sum_{t}%
c_{t}x^{m-\mathrm{rem}\left(  u-tv,m\right)  }f,\ x^{m}-1\right)  &  & \\
&  = \mathrm{rem}\left(  x^{m-\mathrm{rem}(u,m)}f\sum_{t}c_{t}x^{t\ \mathrm{rem}%
\left(  v,m\right)  },\ x^{m}-1\right)  &  & \\
&  = \mathcal{R}_{u}\left(  f\cdot\mathcal{E}_{v}\sum_{t}c_{t}x^{t}\right)  &
\end{align*}

\end{enumerate}
\end{proof}

\subsection{Proof of Theorem~\ref{thm:explicit} (Explicit)}

\begin{lemma}
\label{lemma:phi_psi} We have%
\[
\Phi_{mp}=-\ \Phi_{m}\left(  x^{p}\right)  \ G
\]
where%
\[
G=\Psi_{m}\ \sum_{u\geq0}x^{um}%
\]

\end{lemma}

\begin{proof}
Note%
\begin{align*}
\Phi_{mp}  &  =~~\frac{\Phi_{m}\left(  x^{p}\right)  }{\Phi_{m} } &  &
\text{from }p\nmid m\\
&  =~~\Phi_{m}\left(  x^{p}\right)  \ \frac{\Psi_{m} }{x^{m}-1} &  &
\text{}\\
&  =-\Phi_{m}\left(  x^{p}\right)  \ \Psi_{m} \ \frac{1}{1-x^{m}} &  &
\text{by rearranging}\\
&  =-\ \Phi_{m}\left(  x^{p}\right)  \ \Psi_{m} \ \sum_{u\geq0}x^{um} &  &
\text{by carrying out a formal expansion of }\frac{1}{1-x^{m}}\\
&  =-\ \Phi_{m}\left(  x^{p}\right)  \ G &  &
\end{align*}

\end{proof}

\begin{notation}
\label{notation:gij}Let
\[
G=\Psi_{m}\sum_{u\geq0}x^{um}=\sum_{t\geq0}e_{t}x^{t}
\]
For $0\leq i\leq\varphi(m)-1$ and $0\leq j\leq q,$ let
\[
g_{m,p,i,j}=\sum_{k=0}^{l}e_{ip+mj+k}x^{k}%
\]
where if $j<q\ $then $l=m-1$ else $l=$ $r-1$.
\end{notation}

\begin{lemma}
\label{lem:g} For all $0\leq i\leq\varphi(m)-1$, we have

\begin{enumerate}
\item $g_{m,p,i,0}=\cdots=g_{m,p,i,q-1}=\mathcal{R}_{ir}\Psi_{m}$

\item $g_{m,p,i,q}=\mathcal{T}_{r}\ g_{m,p,i,0}$
\end{enumerate}
\end{lemma}

\begin{proof}
 Let $0\leq j\leq q.$ Let $\Psi_{m}=\sum_{s\geq0}b_{s}x^{s}.$ Since $\deg
\Psi_{m}<m,$ we see immediately that $e_{t}=b_{{\mathrm{rem}}\left(
t,m\right)  }$ for $0\leq t$. We consider two cases:

\begin{enumerate}
\item $j<q$%

\begin{align*}
g_{m,p,i,j}  &  =\sum_{k=0}^{m-1}e_{ip+mj+k}\ x^{k} &  &  ~~\text{from
Notation \ref{notation:gij}}\\
&  =\sum_{k=0}^{m-1}b_{{\mathrm{rem}}(ir+k,m)}\ x^{k} &  &  ~~\text{since
}e_{ip+jm+k}=b_{{\mathrm{rem}}\left(  ip+jm+k,m\right)  }=b_{{\mathrm{rem}%
}(ir+k,m)}\\
&  =\sum_{s=0}^{m-1}b_{s}\ x^{{\mathrm{rem}}(s+i\left(  m-r\right)  ,m)} &  &
~~\text{by re-indexing }k\text{ with }s={\mathrm{rem}}(ir+k,m)\\
&  &  &  ~~\text{which can be easy shown to be a bijection }\mathbb{N}_{\leq
m-1}\rightarrow\mathbb{N}_{\leq m-1}\\
&  &  &  ~~\text{with the inverse map }k={\mathrm{rem}}(s+i\left(  m-r\right)
,m)\\
&  =\sum_{s=0}^{m-1}b_{s}\ {\mathrm{rem}}\left(  x^{s+i\left(  m-r\right)
},\ x^{m}-1\right)  &  &  ~~\text{since }x^{{\mathrm{rem}}(\square
,m)}={\mathrm{rem}}\left(  x^{\square},x^{m}-1\right) \\
&  = \sum_{s=0}^{m-1}b_{s}\mathcal{R}_{ir-s}\cdot1 &  &  \text{from Notation
\ref{not:operation}}\\
&  = \mathcal{R}_{ir}\left(  \mathcal{E}_{1} \mathcal{T}_{m} \Psi_{m}\right)
&  &  \text{from Lemma}~\ref{lem:operation}\\
&  =\mathcal{R}_{ir}\Psi_{m} &  &  \text{since }\deg(\Psi_{m})<m ~~
\end{align*}

\item $j=q$
\begin{align*}
g_{m,p,i,q}  &  =\sum_{k=0}^{r-1}e_{ip+mq+k}\ x^{k} &  &  ~~\text{from
Notation \ref{notation:gij}}\\
&  =\sum_{k=0}^{r-1}b_{{\mathrm{rem}}(ir+k,m)}\ x^{k} &  &  ~~\text{since
}e_{ip+jm+k}=b_{{\mathrm{rem}}\left(  ip+jm+k,m\right)  }=b_{{\mathrm{rem}%
}(ir+k,m)}\\
&  =\mathcal{T}_{r}\ g_{m,p,i,0} &  &  ~~\text{from the second line in the
previous case.}%
\end{align*}

\end{enumerate}
\end{proof}

\begin{lemma}
\label{lem:f_g} For all $0\leq i\leq\varphi(m)-1$ and $0\leq j\leq q$, we have%
\[
f_{m,p,i,j}=-\sum_{s=0}^{i}a_{s}g_{m,p,\left(  i-s\right)  ,j}%
\]
where $\Phi_{m}=\sum_{s\geq0}a_{s}x^{s}.$
\end{lemma}

\begin{proof}
Note%
\begin{align*}
\Phi_{mp}  &  =-\Phi_{m}\left(  x^{p}\right)  \ G &  &  \text{from Lemma
\ref{lemma:phi_psi}}\\
&  =-\left(\sum_{s\geq0}a_{s}x^{sp}\right)~G &  &  \text{from }\Phi_{m}=\sum_{s\geq
0}a_{s}x^{s}\\
&  =-\sum_{s\geq0}a_{s}x^{sp}~~\sum_{k\geq0}e_{k}x^{k} &  &  G=\sum_{k\geq
0}e_{k}x^{k}\\
&  =-\sum_{s\geq0}a_{s}x^{sp}~~\sum_{i\geq0}\sum_{j=0}^{q}g_{m,p,i,j}%
\ x^{jm}\ x^{ip} &  &  \text{from Notation~\ref{notation:gij}}\\
&  =-\sum_{s\geq0}\sum_{i\geq0}\sum_{j=0}^{q}a_{s}\ g_{m,p,i,j}\ x^{jm+\left(
s+i\right)  p} &  &  \text{by collecting the exponents of }x\\
&  =-\sum_{i\geq0}\sum_{\substack{s,i\geq0\\s+\bar{s}=i}}\sum_{j=0}^{q}%
a_{s}g_{m,p,\bar{s},j}\ x^{jm+ip} &  &  \text{by re-indexing }\\
&  =-\sum_{i\geq0}\sum_{s=0}^{i}\sum_{j=0}^{q}a_{s}g_{m,p,\left(  i-s\right)
,j}x^{jm+ip} &  &  \text{by re-indexing and }\bar{s}=i-s\\
&  =-\sum_{i\geq0}\sum_{j=0}^{q}\sum_{s=0}^{i}a_{s}g_{m,p,\left(  i-s\right)
,j}x^{jm+ip} &  &  \text{by changing the summation order}\\
&  =\sum_{i\geq0}\left(  \sum_{j=0}^{q}\left(  -\sum_{s=0}^{i}a_{s}%
g_{m,p,\left(  i-s\right)  ,j}\right)  x^{jm}\right)  x^{ip} &  &  \text{by
grouping }%
\end{align*}
Recall that $\deg~ g_{m,p,i,j}<m.$ Thus%
\[
\deg~\sum_{s=0}^{i}a_{s}g_{m,p,\left(  i-s\right)  ,j}<m
\]
Furthermore $\deg$ $g_{m,p,i,q}<r.$ Recall that $p=qm+r.$ Thus%
\[
\deg\sum_{j=0}^{q}\left(  -\sum_{s=0}^{i}a_{s}g_{m,p,\left(  i-s\right)
,j}\right)  x^{jm}<p
\]
Thus finally from Notation \ref{notation:partition}, we have%
\[
f_{m,p,i,j}=-\sum_{s=0}^{i}a_{s}g_{m,p,\left(  i-s\right)  ,j}%
\]

\end{proof}

\begin{proof}
[Proof of Theorem~\ref{thm:explicit}]We consider two cases:

\begin{enumerate}
\item $j<q$%
\begin{align*}
f_{m,p,i,j}  &  =-\sum_{s=0}^{i}a_{s}g_{m,p,\left(  i-s\right)  ,j} &  &
\text{from Lemma \ref{lem:f_g}}\\
&  =-\sum_{s=0}^{i}a_{s}\mathcal{R}_{(i-s)r}\Psi_{m} &  &  \text{from Lemma
\ref{lem:g}}\\
&  =- \mathcal{R}_{ir} (\Psi_{m} \cdot\mathcal{E}_{r} \mathcal{T}_{i+1}
\Phi_{m}) &  &  \text{from Lemma}~\ref{lem:operation}%
\end{align*}

\item $j=q$
\begin{align*}
f_{m,p,i,q}  &  =-\sum_{s=0}^{i}a_{s}g_{m,p,\left(  i-s\right)  ,q} & \\
&  = -\sum_{s=0}^{i}a_{s}\mathcal{T}_{r}g_{m,p,\left(  i-s\right)  ,0} & \\
&  =-\mathcal{T}_{r}\sum_{s=0}^{i}a_{s}g_{m,p,\left(  i-s\right)  ,0} &
\text{from Lemma \ref{lem:g}}\\
&  =~~\mathcal{T}_{r}f_{m,p,i,0} & \text{from Lemma \ref{lem:f_g}}%
\end{align*}

\end{enumerate}
\end{proof}

\subsection{Proof of Theorem~\ref{thm:intra-str} (Intra-Structure)}

\begin{lemma}
\label{lem:cancel} For an integer $s$, we have $\mathcal{R}_{s}\left(
\Psi_{m}\mathcal{E}_{r}\Phi_{m}\right)  =0$ .

\end{lemma}

\begin{proof}
Note
\begin{align*}
\mathcal{R}_{s}\left(  \Psi_{m}\mathcal{E}_{r}\Phi_{m}\right)   &  =
~~{\mathrm{rem}}(x^{m-\mathrm{rem}(s,m)}\Psi_{m}\Phi_{m}(x^{r}),x^{m}-1) & \\
&  =~~{\mathrm{rem}}(x^{m-\mathrm{rem}(s,m)}\Psi_{m}\Phi_{m}(x^{p}),x^{m}-1)
&  &  \Phi_{m}(x^{p})\equiv_{x^{m}-1}\Phi_{m}(x^{r})\\
&  =~~{\mathrm{rem}}(x^{m-\mathrm{rem}(s,m)}\Psi_{m}\Phi_{m}\Phi_{mp},x^{m}-1)
&  &  \Phi_{m}(x^{p})=\Phi_{m}\Phi_{mp}\\
&  =~~{\mathrm{rem}}( x^{m-\mathrm{rem}(s,m)}\left(  x^{m}-1\right)  \Phi
_{mp},x^{m}-1) &  &  \Psi_{m}\Phi_{m}=x^{m}-1\\
&  =~~0 &  &  {\mathrm{rem}}( x^{m}-1,x^{m}-1)=0\\
&  &
\end{align*}

\end{proof}

\begin{proof}
[Proof of Theorem~\ref{thm:intra-str}]\ 

\begin{enumerate}
\item {\textbf{Repetition:}} From Theorem~\ref{thm:explicit} we see that
$f_{m,p,i,j}$ does not depend on $j$. Hence
\[
f_{m,p,i,0}=\cdots=f_{m,p,i,q-1}%
\]

\item {\textbf{Truncation:}} From Theorem~\ref{thm:explicit} we see that
$f_{m,p,i,q}={\mathrm{rem}} (f_{m,p,i,0}, x^{r})$. Hence
\[
f_{m,p,i,q}=\mathcal{T}_{r}f_{m,p,i,0}%
\]

\item {\textbf{Symmetry:}} From Theorem~\ref{thm:explicit},
\begin{align*}
f_{m,p,i^{\prime},0}  &  =- \mathcal{R}_{i^{\prime}r}(\Psi_{m}\mathcal{E}%
_{r}\mathcal{T}_{i^{\prime}+1}\Phi_{m}) &  & \\
&  =- \sum_{s=0}^{i^{\prime}}a_{s}\mathcal{R}_{i^{\prime}r-sr}\Psi_{m} & \\
&  =- \mathcal{R}_{i^{\prime}r}\Psi_{m}\cdot\mathcal{E}_{r}\sum_{s=0}%
^{i^{\prime}}a_{s}x^{s} &  &  \text{by Lemma}~\ref{lem:operation}\\
&  =- \mathcal{R}_{i^{\prime}r}\Psi_{m}\cdot\mathcal{E}_{r}\left(  \Phi
_{m}-\sum_{s=i^{\prime}+1}^{\varphi(m)}a_{s}x^{s}\right)  &  & \\
&  =~~ \mathcal{R}_{i^{\prime}r}\left(  \Psi_{m}\cdot\mathcal{E}_{r}%
\sum_{s=i^{\prime}+1}^{\varphi(m)}a_{s}x^{s}\right)  &  &  \text{by Lemma
\ref{lem:cancel}}\\
&  =~~ \mathcal{R}_{(\varphi(m)-i-1)r}\left(  \Psi_{m}\cdot\mathcal{E}_{r}%
\sum_{s=\varphi(m)-i}^{\varphi(m)}a_{s}x^{s}\right)  &  &  i^{\prime}%
=\varphi(m)-i-1\\
&  =~~ \mathcal{R}_{-(i+1)r}\left(  x^{\mathrm{rem}(-\varphi(m)r,m)}\Psi
_{m}\cdot\mathcal{E}_{r}\sum_{s=\varphi(m)-i}^{\varphi(m)}a_{s}x^{s}\right)
&  &  \text{by Lemma}~\ref{lem:operation}\\
&  =~~ \mathcal{R}_{-(i+1)r}\left(  \Psi_{m}\cdot\mathcal{E}_{r}%
\sum_{s=\varphi(m)-i}^{\varphi(m)}a_{\varphi(m)-s}x^{\mathrm{rem}\left(
s-\varphi(m),m\right)  }\right)  &  &  a_{s}=a_{\varphi(m)-s}\\
&  =~~ \mathcal{R}_{-(i+1)r}\left(  \Psi_{m}\cdot\mathcal{E}_{r}\sum_{t=0}%
^{i}a_{t}x^{\mathrm{rem}\left(  -t,m\right)  }\right)  &  &  t=\varphi(m)-s\\
&  =~~ \mathcal{R}_{-(i+1)r}\left(  \Psi_{m}\cdot\mathcal{E}_{-r}
\mathcal{T}_{i+1}\Phi_{m}\right)  &  &  \text{by Notation}~\ref{not:operation}%
\\
&  =- \mathcal{R}_{-(i+1)r}\left(  x^{\psi(m)}\Psi_{m}(x^{-1})\cdot
\mathcal{E}_{-r} \mathcal{T}_{i+1}\Phi_{m}\right)  &  &  \Psi_{m}%
=-x^{\psi\left(  m\right)  }\Psi_{m}(x^{-1})\\
&  =- \mathcal{R}_{-r}\mathcal{R}_{-ir}\left(  x^{m-1-(\varphi(m)-1)}\Psi
_{m}(x^{-1})\cdot\mathcal{E}_{-r} \mathcal{T}_{i+1}\Phi_{m}\right)  &  & \\
&  =~~\mathcal{R}_{-r+\varphi(m)-1}\mathcal{R}_{-ir}\left(  x^{m-1}\Psi
_{m}(x^{-1})\cdot\mathcal{E}_{-r} \mathcal{T}_{i+1}\Phi_{m}\right)  &  &
\text{by Lemma}~\ref{lem:operation}\\
&  =~~\mathcal{R}_{-r+\varphi(m)-1}x^{m-1}f_{m,p,i,0}(x^{-1}) &  &  \text{by
Theorem}~\ref{thm:explicit}\\
&  =~~\mathcal{R}_{-r+\varphi(m)-1}\mathcal{F}f_{m,p,i,0} &  & \\
&  &
\end{align*}
Hence $f_{m,p,i^{\prime},0}=\mathcal{R}_{\varphi(m)-1-r}\mathcal{F}%
f_{m,p,i,0}$
\end{enumerate}
\end{proof}

\subsection{Proof of Theorem~\ref{thm:inter-str} (Inter-Structure)}

\begin{proof}
[Proof of Theorem~\ref{thm:inter-str}]\ 

\begin{enumerate}
\item {\textbf{Invariance:}} Recall
\[
f_{m,\tilde{p},i,0} =- \mathcal{R}_{ir}(\Psi_{m}\sum_{s=0}^{i} a_{s} x^{sr})
\]
Thus
\[
f_{m,\tilde{p},i,0}=f_{m,p,i,0}%
\]

\item {\textbf{Semi-invariance:}} Note
\begin{align*}
f_{m,\tilde{p},i,0}  &  =~~\mathcal{R}_{i\tilde{r}}\left(  \Psi_{m}%
\cdot\mathcal{E}_{\tilde{r}}\mathcal{T}_{i+1}\Phi_{m}\right)  &  &  \text{from
Theorem~\ref{thm:explicit} }\\
&  = ~~\mathcal{R}_{-ir}\left(  \Psi_{m}\cdot\mathcal{E}_{-r}\mathcal{T}%
_{i+1}\Phi_{m}\right)  &  &  \text{from Notation}~\ref{not:operation} \text{
and }\tilde{r}=m-r\\
&  = ~~\mathcal{R}_{-ir}\left(  x^{\psi(m)}\ \Psi_{m}(x^{-1})\cdot
\mathcal{E}_{-r}\mathcal{T}_{i+1}\Phi_{m}\right)  & \\
&  = ~~\mathcal{R}_{-ir}\left(  x^{m-1-(\varphi(m)-1)}\ \Psi_{m}(x^{-1}%
)\cdot\mathcal{E}_{-r}\mathcal{T}_{i+1}\Phi_{m}\right)  & \\
&  = ~~\mathcal{R}_{\varphi(m)-1}\mathcal{R}_{-ir}\left(  x^{m-1} \Psi
_{m}(x^{-1})\cdot\mathcal{E}_{-r}\mathcal{T}_{i+1}\Phi_{m}\right)  &  &
\text{by Lemma}~\ref{lem:operation}\\
&  = - \mathcal{R}_{\varphi(m)-1}x^{m-1}f_{m,p,i,0}(x^{-1}) &  &  \text{from
Theorem~\ref{thm:explicit} }\\
&  = -\mathcal{R}_{\varphi(m)-1}\mathcal{F}f_{m,p,i,0} &  &
\end{align*}
Hence
\[
f_{m,\tilde{p},i,0}=-\mathcal{R}_{\varphi(m)-1}\mathcal{F}f_{m,p,i,0}%
\]

\end{enumerate}
\end{proof}

\end{document}